\documentclass[12pt]{amsart}

\usepackage{enumerate}

\usepackage[utf8]{inputenc}
\usepackage[T1]{fontenc}
\usepackage[mathscr]{eucal}

\usepackage[a4paper, twoside=false, vmargin={2cm,3cm}, includehead]{geometry}

\newtheorem{lemma}{Lemma}[section]
\newtheorem{theorem}[lemma]{Theorem}
\newtheorem*{theorem*}{Theorem}
\newtheorem{corollary}[lemma]{Corollary}
\newtheorem{question}{Question}
\newtheorem{proposition}[lemma]{Proposition}
\newtheorem*{proposition*}{Proposition}

\newtheorem*{problem*}{Problem}
\newtheorem*{theoremA}{Khintchine Recurrence Theorem~(\cite{Kh34})}

\theoremstyle{definition}
\newtheorem*{claim*}{Claim}

\newtheorem*{definition}{Definition}
\newtheorem*{example}{Example}

\newtheorem*{remark}{Remark}
\newtheorem*{remarks}{Remarks}

\newcommand{\N}{{\mathbb N}}

\newcommand{\R}{{\mathbb R}}
\newcommand{\T}{{\mathbb T}}
\newcommand{\Z}{{\mathbb Z}}

\newcommand{\B}{{\mathcal B}}
\newcommand{\X}{{\mathcal X}}
\newcommand{\Y}{{\mathcal Y}}

\newcommand{\ve}{\varepsilon}

\newcommand{\norm}[1]{\left\Vert #1\right\Vert}

\begin{document}

\title{Under recurrence in the  Khintchine recurrence theorem}

\author{Michael Boshernitzan}
\address[Michael Boshernitzan]{
 Department of Mathematics, Rice University, Houston,  TX, 77005, USA}
\email{michael@rice.edu}

\author{Nikos Frantzikinakis}
\address[Nikos Frantzikinakis]{University of Crete, Department of Mathematics, Voutes University Campus, Heraklion 71003, Greece} \email{frantzikinakis@gmail.com}
\author{Mate Wierdl}
\address[M\'at\'e Wierdl]{
  University of Memphis, Department of Mathematics,
  Memphis  TN,  38152, USA } \email{wierdlmate@gmail.com}
\email{}

\begin{abstract}
The Khintchine recurrence theorem asserts that in a measure preserving system, for every set $A$ and $\varepsilon>0$, we have $\mu(A\cap T^{-n}A)\geq \mu(A)^2-\varepsilon$ for infinitely many $n\in \mathbb{N}$. We show that there are systems having under-recurrent sets  $A$, in the sense that the inequality $\mu(A\cap T^{-n}A)< \mu(A)^2$ holds for every $n\in \mathbb{N}$. In particular,  all ergodic systems of positive entropy have under-recurrent sets. On the other hand, answering a question of V.~Bergelson, we show that not all mixing systems have under-recurrent sets.
We also study variants of these problems where the previous strict inequality is reversed, and deduce that under-recurrence is a much more rare phenomenon than over-recurrence. Finally, we study related problems pertaining to multiple recurrence and derive some interesting combinatorial consequences.
\end{abstract}




 \maketitle

\section{Introduction and main results}
\subsection{Introduction}
One of the most classic results in ergodic theory is the  Khintchine recurrence theorem which provides a
 quantitative refinement of the celebrated recurrence theorem of Poincar\'e:
\begin{theoremA}
Let $(X,\X,\mu,T)$ be a measure preserving system and $A\in \X$ be a set. Then for every $\varepsilon>0$
we have
$$
\mu(A\cap T^{-n}A)\geq \mu(A)^2-\varepsilon
$$ for infinitely many  $n\in \N$.
\end{theoremA}
By considering mixing systems it is easy to see that   the lower bound $\mu(A)^2$ cannot be in general
improved. It is less clear whether the $\ve$ that appears on the right hand side of Khintchine's estimate is
a necessity or can be removed. This raises the following question:
\begin{question}  Is there a measure preserving system $(X,\X,\mu,T)$ and a set $A\in \X$ such that
$
\mu(A\cap T^{-n}A)<\mu(A)^2
$ holds for every $n\in \N$ ? Can we take this system to be mixing?
\end{question}
We show that the answer to both questions is affirmative (see
 Theorem~\ref{T:S1}). Moreover,
 we construct  examples that answer affirmatively analogous questions pertaining to multiple recurrence
 (see Theorem~\ref{T:S3}).

 Another natural question,  first raised
by V.~Bergelson    in \cite[Problem 1]{Be96},  is whether such constructions can be carried out on every mixing system:
\begin{question} Is it true that for every mixing
 measure preserving system $(X,\X,\mu,T)$ there exists a set $A\in \X$ such that
$\mu(A\cap T^{-n}A)\leq\mu(A)^2$ holds for every $n\in \N$ ?
\end{question}
Rather surprisingly, the answer to this question is negative. In fact, we show (see the remark after
 Theorem~\ref{T:S2}) that  if a system has an under-recurrent set, then it necessarily has a Lebesgue component (these notions are defined in the next section).
Hence,
  if a system has  singular maximal
 spectral type, then for  every set  $A\in \X$ with
  $0<\mu(A)<1$
 we have
$\mu(A\cap T^{-n}A)>\mu(A)^2$ for infinitely many $n\in \N$. Systems with singular maximal
 spectral type include all  rigid systems and several (potentially all, as conjectured in \cite{Kl97})  rank one transformations. So, in a sense, systems that have under-recurrent sets are rather rare.

 Another interesting fact is that although there are examples of over-recurrent sets for which the sequence
 $\mu(A\cap T^{-n}A)-\mu(A)^2$  converges to $0$ arbitrarily slowly (see Theorem~\ref{T:S4}), for  under-recurrent sets some stringent conditions apply which force this sequence to always be    (absolutely) summable (see  the remark after Theorem~\ref{T:F1}). In a sense, sets do not like to be under-recurrent.

In the next section,  we give the precise statements of the results alluded to in  the previous discussion and
also give several  relevant  refinements and combinatorial consequences.

\section{Main results}
To facilitate our discussion we first introduce some notation.

 A {\em measure preserving system}, or simply {\em a system}, is a quadruple $(X,\X,\mu,T)$
where $(X,\X,\mu)$ is a probability space and $T\colon X\to X$ is a  measure preserving transformation.
Throughout, {\em all  functions are assumed to be real valued} and with $Tf$ we denote the composition $f\circ T$.
\begin{definition}
Let $(X,\X, \mu, T)$ be a system. We say that
\begin{enumerate}
\item The set $A\in \X$ is {\em under-recurrent} if $0<\mu(A)<1$ and
$$
\mu(A\cap T^{-n}A)\leq \mu(A)^2\quad \text{ for every } n\in \N.
$$
\item The set $A\in \X$ is {\em over-recurrent} if $0<\mu(A)<1$ and
$$
\mu(A\cap T^{-n}A)\geq \mu(A)^2\quad \text{ for every } n\in \N.
$$
\item The function $f\in L^2(\mu)$ is {\em under-recurrent} if it is non-constant and
$$
\int f \cdot T^n f \, d\mu\leq\Big(\int f\, d\mu\Big)^2  \quad  \text{ for every } n\in \N.
$$

\item The function $f\in L^2(\mu)$ is {\em over-recurrent} if it is non-constant and
$$
\int f \cdot T^n f \, d\mu\geq\Big(\int f\, d\mu\Big)^2  \quad  \text{ for every } n\in \N.
$$

\item If we have strict inequality we say that the set or the function is {\em strictly} under-recurrent and over-recurrent respectively.
\end{enumerate}
\end{definition}
It is not hard to see  using the von Neumann ergodic theorem that for every system  $(X,\X,\mu,T)$ and function $f\in L^2(\mu)$  we  have
$
\limsup_{n\to \infty} \int f \cdot T^n f \, d\mu\geq\big(\int f\, d\mu\big)^2.
$
If in addition the system is ergodic, then
$\liminf_{n\to \infty} \int f \cdot T^n f \, d\mu\leq\big(\int f\, d\mu\big)^2.
$
This explains why we use the constants $\mu(A)^2$ and  $(\int f\, d\mu)^2$  in the above definitions.  Furthermore, we note that  over-recurrent functions are not hard to come by; for instance,   on a cartesian product system for each zero mean $f\in L^2(\mu)$ the function $f\otimes f$ is over-recurrent.
\subsection{Under and over recurrent sets} We start by stating some results related to  under and over recurrent sets.

Our first result gives an affirmative answer to Question~1 in the introduction.
\begin{theorem}\label{T:S1}
There exists a mixing system that has strictly under-recurrent and  strictly over-recurrent sets.
\end{theorem}
\begin{remark}
We give two proofs of this result. One uses an explicit construction on Bernoulli systems and implies that every positive entropy system has under and over recurrent sets (see Theorem~\ref{T:posentropy}). The other is an indirect construction that is somewhat more versatile (see Section~\ref{SS:Proof}); for a large class of systems we establish the existence of under and over recurrent  functions with values on $[0,1]$ (see Proposition~\ref{P:+-}) and we then deduce using Proposition~\ref{T:Correspondence} the existence of under and over  recurrent sets on different systems. Using this second method we can  prove
 more delicate results like the following:
 For any partition $\N=S_+\cup S_-$, there exist a mixing system and a set $A$, such that
$\mu(A\cap T^{-n}A)>\mu(A)^2$  for  every $n\in S_+$ and $\mu(A\cap T^{-n}A)<\mu(A)^2$  for  every $n\in S_-$.
\end{remark}


V.~Bergelson   asked in \cite[Problem 1]{Be96} whether every mixing system has a strictly under-recurrent set and whether it has  a strictly over-recurrent set. We show that the answer to the first question (and thus to Question 2 in the introduction) is negative. The second question remains open; see Problem~1 in Section~\ref{SS:Problems}.
\begin{theorem}\label{T:S2}
There exists a mixing system with no under-recurrent sets.
\end{theorem}
\begin{remark} We show something stronger:  Any system with singular maximal spectral type (with respect to the Lebesgue measure on $\T$) has no under-recurrent functions  and in fact, for every non-constant
$f\in L^2(\mu)$ we have  $\int f\cdot T^nf\, d\mu>(\int f\, \mu)^2$  for infinitely many $n\in \N$ (see the remark after Theorem~\ref{T:F1}). Examples of mixing systems with singular maximal spectral type include
the Chacon map and several other rank one transformations \cite{Bo93, Kl97}  as well as certain  Gaussian systems.
\end{remark}

It is natural to inquire whether variants of Theorem~\ref{T:S1} hold that
  deal  with the concept of  multiple under-recurrence.
It is not hard to prove the following extension of the Khintchine recurrence theorem: For every system $(X,\X,\mu,T)$, set $A\in \X$, and every $\varepsilon>0$,  we have
  for every $\ell \in \N$ that
  $$
\mu(A\cap T^{-n_1}A\cap \cdots \cap T^{-n_\ell}A)\geq \mu(A)^{\ell+1}-\varepsilon
 $$
 for infinitely many distinct $n_1,\ldots, n_\ell\in \N$.
  The next result shows that the $\varepsilon$ in the above estimate cannot be removed.
\begin{theorem}\label{T:S3}
  For every $d \in \N$ there exist a multiple mixing system\footnote{A system  $(X,\X,\mu,T)$
  is multiple mixing if for every $\ell\in\N$ and $f_0, f_1, \ldots, f_\ell\in L^\infty(\mu)$ one has $\lim_{n_1,\ldots, n_\ell} \int f_0\cdot  T^{n_1}f_1\cdot \ldots \cdot T^{n_\ell}f_\ell\, d\mu=\prod_{i=1}^\ell\int f_i\, d\mu$ where  the limit is taken when  $\min\{n_1, n_2-n_1,\ldots, n_\ell-n_{\ell-1}\}\to \infty$.}
    $(X,\X,\mu,T)$ and a set $A\in \X$ with $0<\mu(A)<1$ and such that for $\ell=1,\ldots, d$ we have
  $$
  \mu(A\cap T^{-n_1}A\cap \cdots \cap T^{-n_\ell}A)<\mu(A)^{\ell+1}
  $$
  for  all distinct $n_1,\ldots, n_\ell\in \N$.

  A similar result also holds with the strict inequality reversed.
\end{theorem}

Finally, we remark that if we do not impose any ergodicity assumptions on the system, we can prove a  variant of Theorem~\ref{T:S1}
which gives more information about the possible values of the difference
$$
d_A(n):=\mu(A\cap T^{-n}A)-\mu(A)^2.
$$
We say that the  sequence $(a_n)_{n\in \N}$ {\em decreases convexly to $0$} if  it converges to $0$ and  satisfies $a_{n-1}+a_{n+1}-2a_n\geq0$ for every $n\geq 2$. Note that then $(a_n)_{n\in\N}$ is necessarily non-negative and  decreasing.
\begin{theorem}\label{T:S4}
 Let $(a_n)_{n\in \N}$ be a real valued sequence. If  either $\sum_{n=1}^\infty|a_n|<1/16$, or $(a_n)_{n\in \N}$ decreases convexly to $0$ and satisfies $a_1\leq 1/8$,
 then there exist a system $(X,\B,\mu,T)$ and a set $A\in \X$ such that
$d_A(n)=a_n$ for every   $n \in \N.$
\end{theorem}
\begin{remarks}
$\bullet$
We show something stronger: If $\sigma$ is a symmetric probability measure on $\T$, then there exist a system $(X,\B,\mu,T)$ and a set $A\in \X$ such that
$
d_A(n)=\frac{1}{8}\,  \widehat{\sigma}(n)$ for every   $n \in \N.$

$\bullet$ When it comes to under-recurrence, the first hypothesis
is not so severe; if  $A$ is an under-recurrent set, then   the proof of  Proposition~\ref{P:key}  gives
$
\sum_{n=1}^\infty\big| d_A(n)| <1/2$.
\end{remarks}

\subsection{Under and over recurrent functions}
Next,  we state results related to under and over recurrence properties of functions. For a particular class of systems, which we define next, it is possible to guarantee the existence of under and over recurrent functions.
\begin{definition}
We say that the system  $(X,\mathcal{X}, \mu,T)$ has
a {\em Lebesgue component} if there exists a function  $f\in L^2(\mu)$ with spectral measure equal to the Lebesgue measure on $\T$,
that is,  satisfies
$\norm{f}_{L^2(\mu)}=1$ and $\int f \cdot T^n f \, d\mu=0$ for every $n\in \N$.
\end{definition}
Systems  having a Lebesgue component include ergodic nilsystems that are not rotations \cite[Theorem~4.2]{AGH63},
 \cite[Proposition~2.1]{HKM14},
and positive entropy systems  \cite{R60}.

The next result includes a convenient characterization of systems that have under-recurrent functions (it is the equivalence $(i) \Longleftrightarrow (iii)$).


\begin{theorem}\label{T:F1}
Let $(X,\mathcal{X}, \mu,T)$ be a system. Then the following  are equivalent:
\begin{enumerate}[(i)]

\item The system has  an  under-recurrent function.

\item The system has  a strictly under-recurrent function.

\item The system has a  Lebesgue component.

 \item For every  non-negative  $\phi\in L^1(m_\T)$
  there exists   $g\in L^2(\mu)$ such that
  $$
  \int g\, d\mu=0 \ \text{ and } \ \int g\cdot T^ng\, d\mu=\widehat{\phi}(n) \ \ \text{ for every }  n\in \N.
  $$

\end{enumerate}
\end{theorem}
\begin{remark}
Our argument shows that  if $f\in L^2(\mu)$ is
under-recurrent, then the spectral measure of $f$ is absolutely continuous with respect to the Lebesgue measure  on $\T$ and
$\sum_{n=1}^\infty | \int f \cdot T^n f \, d\mu -(\int f\, d\mu)^2|<\infty.$

\end{remark}
We  deduce from the previous  result the following:
\begin{corollary}\label{C:F1}
Suppose that the system $(X,\X,\mu,T)$ has  an under-recurrent function.  Then it also  has a strictly over-recurrent function.
Furthermore, if the real valued sequence
$(a_n)_{n\in\N}$ is either  absolutely summable  or  decreases convexly to $0$, then there exists  $f\in L^2(\mu)$ such that
 $d_f(n)=a_n$ for every $n\in \N$, where $d_f(n):=\int f\cdot T^nf \, d\mu-(\int f\, d\mu)^2$.
\end{corollary}
Finally, we record a  variant of Theorem~\ref{T:F1} which can be used (via Proposition~\ref{T:Correspondence} below) in order to deduce the existence of under or over recurrent  sets for certain classes of systems.
\begin{theorem}\label{T:F2}
If the system  $(X,\mathcal{X}, \mu,T)$ has a   Lebesgue (spectral) measure realized by a bounded function, then it has a strictly
under-recurrent function with values in $[0,1]$ and a strictly
over-recurrent function with values in $[0,1]$.
\end{theorem}
\begin{remark}
For a stronger statement see Proposition~\ref{P:+-}. Note also that by a theorem of V. M. Alexeyev \cite{Al82} our assumption is satisfied for every system that has Lebesgue maximal spectral type.
\end{remark}
\subsection{Combinatorial consequences}
Finally, we use under and over recurrence properties  of ergodic measure preserving systems in order to deduce
some   combinatorial consequences. In what follows, with
$d(E)$ and $\bar{d}(E)$  we denote the  density and the upper density of a set $E\subset \N$ respectively. Whenever we write $d(E)$ we implicitly assume that the density of the set $E$ exists.
Using the remark made immediately after Theorem~\ref{T:S1} we deduce the following:
\begin{theorem}\label{T:c1}
For any partition $\N=S_+\cup S_-$ there exists  $E\subset \N$, such that
\begin{enumerate}
\item $d(E\cap (E-n))>d(E)^2$  \ for  every \ $n\in S_+$;
\item $d(E\cap (E-n))<d(E)^2$ \ for  every  \ $n\in S_-$;
\item  $\lim_{n\to\infty}d(E\cap (E-n))=d(E)^2$.
\end{enumerate}
\end{theorem}
Note that even the existence of a set of positive integers $E$ that satisfies Property $\text{(ii)}$  for $S_-=\N$
seems non-trivial to establish.

It can be shown (one way is to deduce this from the corresponding ergodic statement via the Furstenberg correspondence principle) that for every  set $E\subset \N$ and every $\ell\in \N$ and $\varepsilon>0$, we have
$$
 \bar{d}(E\cap (E-n_1)\cap\cdots\cap (E-n_\ell))\geq \bar{d}(E)^{\ell+1} -\varepsilon
$$
for infinitely many distinct $n_1,\ldots, n_\ell\in \N$.
Using   Theorem~\ref{T:c1} we show that for every $\ell\in \N$ the $\varepsilon$ in the above statement cannot in general be removed.
\begin{theorem}\label{T:c2}
For every $r \in \N$ there exists  a set of positive integers $E$ such that
for  $\ell=1,\ldots, r$ we have
\begin{enumerate}
\item $d(E\cap (E-n_1)\cap\cdots\cap (E-n_\ell))<d(E)^{\ell+1}$
for all  distinct $n_1,\ldots, n_\ell\in \N$;
\smallskip

\item $\lim_{ n_1,\ldots, n_\ell }d(E\cap (E-n_1)\cap\cdots\cap (E-n_\ell))=d(E)^{\ell+1}$,
\end{enumerate}
where in $\text{(ii)}$ the limit is taken when  $\min\{n_1, n_2-n_1,\ldots, n_\ell-n_{\ell-1}\}\to \infty$.

Moreover, there exists $E\subset \N$ that satisfies Property $(i)$ with the  strict inequality reversed, and Property $(ii)$.
\end{theorem}

\subsection{Open problems}\label{SS:Problems}
 Theorem~\ref{T:S2}  asserts that there exist mixing systems with no under-recurrent functions and   the key to our construction was the fact that
  every under-recurrent   function  has
 spectral measure absolutely continuous with respect to the Lebesgue measure.
This   property  is not shared by over-recurrent functions (see the example in Section~\ref{SS:ex2}), and
in fact,  constructing weakly mixing systems that have no  over-recurrent functions (or  sets) turns out to be much harder, perhaps even impossible.\footnote{On the other hand, it is not hard to verify that  ergodic Kronecker systems do not have over-recurrent functions.}
 This leads to the following question (a variant of it was also  asked by V.~Bergelson    in \cite[Problem 1]{Be96}):

\medskip
\noindent {\bf Problem 1.}
{\em Does every  (weakly) mixing  system have an  over-recurrent set or  a bounded over-recurrent function?}
\medskip

  C.~Badea and  V.~ M\"{u}ller showed in  \cite{BM09} that every mixing system has a strictly  over-recurrent function
 in $L^2(\mu)$.  In fact, on every mixing system, for every sequence $(a_n)_{n\in \N}$ of positive  reals that converges to $0$, and every $\varepsilon>0$, there exists a function $f\in L^2(\mu)$ with $\norm{f}_{L^2(\mu)}\leq \sup_n a_n+\varepsilon$ and such that $\int f\cdot T^nf\, d\mu>a_n$ for every $n\in \N$.

It is possible (via Proposition~\ref{T:Correspondence} below)  to transfer recurrence  properties
 of a function with values in $[0,1]$ in a given system, to recurrence properties of a set in a different  system.
 It is not clear  whether a similar construction can take place without changing the system.

\medskip

\noindent {\bf Problem 2.}
{\em
If a system has an under-recurrent function with values in $[0,1]$ does it always have an under-recurrent set?}

\subsection{ Acknowledgments} We would like to thank V.~Bergelson for bringing to our attention the article \cite{BM09},  B.~Host  for a construction used  in Section~\ref{SS:ex2},
and the referee for useful comments.

\section{Under and over recurrent functions}
In this section, we give the proofs of the results
 pertaining to   under and over recurrence properties of functions;  in the next section,
    we use some of these results in order to deduce analogous properties for sets.

\subsection{Proof of Theorem~\ref{T:F1} and Corollary~\ref{C:F1}}
In this subsection, it
is convenient to think of a correlation sequence $(\int f\cdot T^nf\, d\mu)_{n\in\N}$
 as the sequence of Fourier coefficients 
 of the  spectral measure
$\sigma_f$ of the function $f$. This way, if $\int f\, d\mu=0$, then
under or over recurrence properties of a function $f\in L^2(\mu)$ correspond to statements about
the sign of the sequence $(\widehat{\sigma_f}(n))_{n\in\N}$. Keeping this in mind,
the key to the proof of Theorem~\ref{T:F1} is the following simple Fourier analysis result:
\begin{proposition}\label{P:key}
Let $\sigma$ be a  probability measure  on $\T$ such that  $\text{Re}(\widehat{\sigma}(n))\leq 0$ for every  $n\in \N$. Then
$$
\sum_{n=1}^\infty |\text{Re}(\widehat{\sigma}(n))|\leq 1/2.
$$
\end{proposition}
\begin{remark}
 Our argument shows that if $\sigma$ is a symmetric probability measure  on $\T$ with a convergent sum of positive Fourier coefficients, then
$
-\sum_{n\in F_-}\widehat{\sigma}(n)\leq \sum_{n\in F_+}\widehat{\sigma}(n)+\frac{1}{2},
$
where $F_+:=\{n\in \N\colon \widehat{\sigma}(n)> 0\}$ and $F_-:=\{n\in \N\colon \widehat{\sigma}(n)< 0\}$.
\end{remark}
\begin{proof}[Proof of Proposition \ref{P:key}]
For every $N\in \N$ we have
\begin{align*}\label{E:a1}
0\leq \int \Big|\sum_{n=1}^N e(nt)\Big|^2\,  d\sigma(t)&=N+2\,\text{Re}\!\sum_{1\leq m<n\leq N} \int e((n-m)t)\, d\sigma(t)=\\
 &=N+2\,\sum_{k=1}^N (N-k)\,\text{Re}(\widehat{\sigma}(k))=N+2\,\sum_{k=1}^{N-1} S_k,\notag
\end{align*}
where $S_k=\sum_{n=1}^k\text{Re} (\widehat{\sigma}(n))$ and  $e(t):=e^{2\pi i t}$.
We conclude that
\begin{equation}\label{E:a3}
\frac{1}{2}+\frac{1}{2N}+\frac{1}{N}\sum_{n=1}^N S_n\geq 0 \quad \text{for every } N\in \N.
\end{equation}

Since $(-S_n)$ is a non-decreasing sequence in $[0,\infty)$, it has a limit
\[
L=\sum_{n=1}^\infty |\text{Re}(\widehat{\sigma}(n))|\in[0,\infty].
\]
Taking $N\!\to\!\infty$ in \eqref{E:a3} results in the inequality
$\frac12-L\geq0$, completing the proof.
\end{proof}

\begin{corollary}\label{C:Leb}
Let $\sigma$ be a symmetric probability measure on $\T$ such that
$\widehat{\sigma}(n)\leq 0$ for every $n\in \N$. Then $\sigma$ is equivalent to the Lebesgue measure on $\T$.
\end{corollary}
\begin{proof}
Using Proposition~\ref{P:key} we get that
$$
\sum_{n=1}^\infty |c_n|\leq 1/2
$$
where $c_n:=\widehat{\sigma}(n)$, $n\in \N$. Then
$d\sigma= \phi\, dm_\T$ where
$$
\phi(t):=1+2\sum_{n=1}^\infty c_n \cos(2\pi nt).
$$
We have that $\phi(t)\geq 0$ for all $ t\in \R$, with  equality only if
$$\sum_{n=1}^\infty |c_n \cos(2\pi nt)|=\sum_{n=1}^\infty |c_n|=1/2.
$$
 This can happen   only for finitely many $t\in \T$. It follows that $\sigma$ is equivalent to the measure $m_\T$.
\end{proof}
We also need the following classic result from the spectral theory of unitary operators:
\begin{proposition}\label{P:spectral}
Let $(X,\X,\mu,T)$ be a system, $f\in L^2(\mu)$ be a function, and let
 $\rho$ be a finite measure that is absolutely continuous with respect to $\sigma_f$. Then
there exists    $g\in L^2(\mu)$ with $\sigma_g=\rho$.
\end{proposition}
\begin{proof}
We have that $d\rho= \phi\,  d\sigma_f$ for some non-negative function  $\phi\in L^1( \sigma_f)$.
Then  $\phi=\psi^2$ for some  real valued $\psi\in L^2( \sigma_f)$.
 Let $g:=\psi(T)f$ (see \cite[Corollary~2.15]{Qu10}  for the definition of the operator $\psi(T)$). Then $g$ is  real valued, $g\in L^2(\mu)$,  and
$d\sigma_g=\psi^2\, d\sigma_f=d\rho$, that is, $\sigma_g=\rho$.
\end{proof}

\begin{proof}[Proof of Theorem~\ref{T:F1}]
We show that $(i)\Longrightarrow (iii)$. Suppose that $f$ is an under-recurrent function.  Then the
spectral measure of the real valued function $g:=f-\int f\, d\mu$ is symmetric, not identically $0$ (since our standing assumption is that $f$ is non-constant),  and satisfies
$$
\widehat{\sigma_g}(n)=\int f\cdot T^nf\, d\mu -\Big(\int f\, d\mu\Big)^2\leq 0 \ \text{ for every } n\in \N.
$$
 Hence,  by Corollary~\ref{C:Leb}, the measure $\sigma_g$ is equivalent to $m_\T$. We deduce from this and Proposition~\ref{P:spectral},  that  there exists a function $h\in L^2(\mu)$ with $\sigma_h=m_\T$. Hence, the system has a Lebesgue component.

We show that $(iii)\Longrightarrow (iv)$. Let $f\in L^2(\mu)$ have spectral measure
 $\sigma_f=m_\T$ and  let $\phi\in L^1(m_\T)$ be non-negative. By Proposition~\ref{P:spectral} there exists  a  zero mean
 $g\in L^2(\mu)$ such that $d\sigma_g=\phi \, dm_\T$. Then
 $$
 \int g \cdot T^n g\, d\mu=\widehat{\sigma}_g(n)=\widehat{\phi}(n) \ \ \text{ for every } n\in \N.
$$
Furthermore, since $\int g \cdot T^n g\, d\mu=\widehat{\phi}(n)\to 0$ as $n\to \infty$, we deduce from the ergodic theorem that
$$
\Big(\int g\, d\mu\Big)^2\leq \lim_{N\to \infty}\frac{1}{N}\sum_{n=1}^N\int g\cdot T^ng\, d\mu=0.
$$
Hence, $\int g\, d\mu=0$.

We show that $(iv)\Longrightarrow (ii)$.  We apply $(iv)$ for the non-negative (real valued) function $\phi\in L^\infty(\T)$ defined by
$$
\phi(t):=1-\sum_{n=1}^\infty \frac{1}{2^{n+2}}(e(nt)+e(-nt)).
$$
We get that there exists $g\in L^2(\mu)$ such that $\sigma_g=\phi \, dm_\T$. Then
$$
\int g\cdot T^ng\, d\mu=-\frac{1}{2^{n+2}}<0=\Big(\int g\, d\mu\Big)^2 \ \text{ for every} \ n\in \N.
$$
Hence, the function $g$ is strictly under-recurrent.

Finally, the implication $(ii)\Longrightarrow (i)$ is obvious.
\end{proof}

\begin{proof}[Proof of Corollary~\ref{C:F1}] The proof follows by combining the implication
$(i)\Longrightarrow (iv)$ of Theorem~\ref{T:F1} and the fact that under the stated assumptions there exists a non-negative even function $\phi\in L^1(m_\T)$ such that $\widehat{\phi}(n)=a_n$ for every $n\in \N$. Indeed, if $(a_n)_{n\in \N}$ decreases convexly to $0$, then this is a classic result (see for example  \cite[Theorem~4.1]{Ka04}); on the other hand,  if $\sum_{n=1}^\infty|a_n|=A$, we let  $\phi(t):=2A+\sum_{n=1}^\infty a_n (e(nt)+e(-nt))$.
\end{proof}

\subsection{Proof of Theorem~\ref{T:F2}}
We prove a more general result:
\begin{proposition}\label{P:+-}
Let $(X,\X,\mu,T)$ be a system  that has a Lebesgue component defined by an $L^\infty(\mu)$ function.
 Then for any partition $\N=S_+ \cup S_{-}$  there exists $f\in L^\infty(\mu)$, with values in $[0,1]$, such that
\begin{enumerate}\label{E:+-}
\item \label{E:+} $\int f\cdot T^nf\, d\mu>(\int f\, d\mu)^2$ \quad for  every \ $n\in S_+$;
\item \label{E:-} $\int f\cdot T^nf\, d\mu< (\int f\, d\mu)^2$\quad   for every \  $n\in S_-.$
\end{enumerate}
\end{proposition}
\begin{proof}
First note that it suffices to find  $f\in L^\infty(\mu)$ that satisfies Properties~\eqref{E:+} and \eqref{E:-} without imposing any other restriction on its range. Indeed, then $\tilde{f}:=\frac{1}{{2\norm{f}_\infty}}(\norm{f}_\infty+f)$ still satisfies Properties~\eqref{E:+} and \eqref{E:-}  and takes values in $[0,1]$.

Our assumptions imply that there exists a function  $g\in L^\infty(\mu)$ such that
$$
\int g\cdot T^ng\, d\mu=0 \ \text{  for every } n\in \N \ \text{ and }\ \norm{g}_{L^2(\mu)}=1.
$$
Note that then  $\int g\, d\mu=0$; indeed,   the ergodic theorem implies that
$$
\Big(\int g\, d\mu\Big)^2\leq \lim_{N\to\infty}\frac{1}{N}\sum_{n=1}^N\int g\cdot T^ng\, d\mu=0.
$$

Now let $(a_n)_{n\in \N}$ be a sequence of real numbers such that
\begin{enumerate}[(i)]
\item
$(|a_n|)_{n\in \N}$ is decreasing;

 \item
  $a_k>0$ for $k\in S_+$ and $a_k<0$ for $k\in S_-$;

 \item $\sum_{k=1}^\infty|a_k| <1$.
 \end{enumerate}
 Consider the  function
$$
f:=g+\sum_{k=1}^\infty a_k T^kg.
$$
Then $f\in L^\infty(\mu)$ and $\int f\, d\mu=0$. We set  $a_0:=1$.
Then
$$
\int f\cdot T^nf\, d\mu=\sum_{k,l\geq 0} a_k\cdot a_{l} \int T^lg\cdot T^{k+n}g\, d\mu \ \ \text{ for every } \ n\in \N.
$$
Since, by assumption,
 $\{g, Tg, T^2g,\ldots\}$ is an orthonormal set, we deduce that
$$
\int f\cdot T^nf\, d\mu=a_n+\sum_{k=1}^\infty a_k\cdot a_{k+n} \ \ \text{ for every } \ n\in \N.
$$
Hence, using Properties $\text{(i)}$-$\text{(iii)}$ we get for $n\in S_+$ that
$$
\int f\cdot T^nf\, d\mu\geq a_n-\sum_{k=1}^\infty |a_k|\cdot |a_{k+n}|\geq
a_n-a_n\sum_{k=1}^\infty |a_k|>0;
$$
and for $n\in S_-$ that
$$
\int f\cdot T^nf\, d\mu\leq a_n+\sum_{k=1}^\infty |a_k|\cdot |a_{k+n}|\leq
a_n-a_n\sum_{k=1}^\infty |a_k|<0.
$$
This completes the proof.
\end{proof}

\section{Under and over recurrent sets}

\subsection{From functions to sets}
We will use the following ``correspondence principle'' in order to translate statements about correlation sequences
of functions with
values on the interval $[0,1]$ to statements about correlation sequences of sets.
\begin{proposition}\label{T:Correspondence}
Let $(X,\X,\mu,T)$ be a system and $f\in L^\infty(\mu)$ be a function that takes values in $[0,1]$. Then there
exist an invertible  system $(Y,\Y,\nu, S)$ and a set $A\in \Y$ such that
$$
\nu(S^{-n_1}A\cap\cdots \cap S^{-n_\ell}A)=\int  T^{n_1}f \cdot \ldots\cdot T^{n_\ell}f \, d\mu
$$
holds for every $\ell\in \N$ and  all distinct non-negative integers $n_1,\ldots, n_\ell$.

 Moreover,  if the system $(X,\X,\mu,T)$ is ergodic, weak-mixing,  mixing, or multiple-mixing, then so is the system $(Y,\Y,\nu, S)$.
\end{proposition}

\begin{proof}
In the sequence space $Y:=\{0,1\}^\Z$ we
denote the cylinder  sets by
$$
[\epsilon_m \epsilon_{m+1}\ldots \epsilon_n]:=\{x=(x_m)_{m\in\Z}\colon x_{m}=\epsilon_m,x_{m+1}=\epsilon_{m+1}, \ldots, x_n=\epsilon_n\},
$$
where  $m,n\in \Z$, $m\leq n$, and $\epsilon_i\in \{0,1\}$ for  $i\in \Z$. We let
$$
f_0:=f, \quad f_1:=1-f,
$$
and define the measure $\nu$ on cylinder sets by
\begin{equation}\label{E:defnu}
\nu([\epsilon_0 \epsilon_1 \ldots  \epsilon_n]):=\int f_{\epsilon_0}\cdot Tf_{\epsilon_1}\cdot \ldots\cdot
T^nf_{\epsilon_n}\, d\mu.
\end{equation}
We extend $\nu$  to all cylinder sets in a stationary way.
The consistency conditions of Kolmogorov's extension
theorem are satisfied, thus $\nu$ extends to a stationary measure on the Borel $\sigma$-algebra of the sequence space $Y$.
If
$$
A:=\{x\in Y\colon x_0=1\},
$$
and $S$ is the shift transformation on $Y$, then
$$
\nu( S^{-n_1}A\cap\cdots \cap S^{-n_\ell}A)=\int  T^{n_1}f \cdot \ldots\cdot T^{n_\ell}f \, d\mu
$$
holds for every $\ell\in \N$ and  all distinct non-negative integers $n_1,\ldots, n_\ell$.

Finally, suppose that the system $(X,\X,\mu,T)$ is mixing (ergodicity, weak-mixing, and multiple-mixing can be treated similarly). Let
$$
A:=[\epsilon_0 \epsilon_{1} \ldots  \epsilon_k], \quad B:=[\tilde{\epsilon}_{0} \tilde{\epsilon}_{1} \ldots  \tilde{\epsilon}_{l}]
$$
be two cylinder sets, and let
$$
g:=f_{\epsilon_0}\cdot Tf_{\epsilon_1} \cdot \ldots\cdot T^kf_{\epsilon_k}, \quad
h:=f_{\tilde{\epsilon}_0}\cdot Tf_{\tilde{\epsilon}_1} \cdot \ldots\cdot T^{l}f_{\tilde{\epsilon}_{l}}.
$$
Then for $n>\max\{k,l\}$, by the defining property of $\nu$ (see \eqref{E:defnu}),   we have
$$
\nu(A\cap S^{-n}B)=\int g\cdot T^nh\, d\mu\to \int g\, d\mu\cdot \int h\, d\mu=\nu(A)\cdot \nu(B).
$$
By stationarity, we get a similar statement for any two cylinder sets, and by density for all Borel subsets of $Y$. This proves  the asserted mixing property for the system $(Y,\Y,\nu,S)$.
\end{proof}

\subsection{Proof of Theorems~\ref{T:S1} and  \ref{T:S2}}\label{SS:Proof} We are now ready to prove results about
 under and  over recurrent  sets.
\begin{proof}[Proof of Theorem~\ref{T:S1}]
Let $(X,\X,\mu,T)$ be a mixing system with a Lebesgue component defined by a bounded function (for example a Bernoulli system). Then by Theorem~\ref{T:F2} there exists a strictly under-recurrent function $f$ with values in $[0,1]$. By Proposition~\ref{T:Correspondence} there exist a mixing system $(Y,\Y,\nu,S)$ and a set $A\in \Y$ with $\nu(A)=\int f\, d\mu$ (then $0<\nu(A)<1$ since $f$ is non-constant), and such that for every $n\in \N$ we have
$$
\nu(A\cap S^{-n}A)=\int f\cdot T^nf\, d\mu<\Big(\int f\, d\mu\Big)^2=\nu(A)^2.
$$
Hence, the set $A$ is strictly under-recurrent.

A similar argument proves the existence of a strictly over-recurrent set $B$ in some other mixing system $(Y,\Y,\nu,S)$.

To get a single system with a strictly under-recurrent and a strictly over-recurrent set, we consider the direct product $(X\times Y,\X\times \Y,\mu\times \nu,T\times S)$ of the two systems. This  system is still mixing, the set $A\times Y$ is strictly under-recurrent, and the set $X\times B$ is strictly over-recurrent.
\end{proof}
To prove  the  statement in the remark following Theorem~\ref{T:S1} we repeat the previous argument replacing  Theorem~\ref{T:F2} with Proposition~\ref{P:+-}.

\begin{proof}[Proof of Theorem~\ref{T:S2}]
It is known that there exist mixing systems with singular maximal spectral type (see the remark after Theorem~\ref{T:S2}).
Then by Theorem~\ref{T:F1} any such  system does not have under-recurrent functions, and as a consequence,  does not have under-recurrent sets.
\end{proof}
The remark after Theorem~\ref{T:S2} follows in a similar fashion since if $f\in L^2(\mu)$ satisfies $\int f\cdot T^nf\, d\mu\leq (\int f\, d\mu)^2$ for all large enough $n\in\N$, then the argument used in the proof of Corollary~\ref{C:Leb} shows that the spectral measure of $f$ is absolutely continuous with respect to the Lebesgue measure on $\T$.

\subsection{Proof of Theorem~\ref{T:S4}} Recall that for $f\in L^2(\mu)$ and $n=0,1,2,\ldots$ we define
$d_f(n):= \int f \cdot T^n f \, d\mu -\big(\int f\, d\mu\big)^2$. The proof of Theorem~\ref{T:S4}
is based on the following result:
\begin{proposition}\label{P:01}
Let $\sigma$ be a symmetric probability measure on $\T$. Then there exist a system $(X,\B,\mu,T)$ and
a function $f\in L^\infty(\mu)$ that takes values in $[0,1]$ and   such that
$$
d_f(n)=\frac{1}{8}\cdot \widehat{\sigma}(n) \ \ \text{ for  }  n=0,1,2\ldots.
$$
\end{proposition}
\begin{proof}
 Consider the system $(\T^2, \B_{\T^2}, \mu, T)$ where
$$
T(x,y):=(x,y+x) \pmod{1}
$$
and $\mu=\sigma\times m_\T$ ($m_\T$ is the Haar measure on $\T$). Note that    $T$  preserves the measure $
\mu$.
Let
$$
f(x,y):=\frac{1}{2}(1+\cos(2\pi y))=\frac{1}{2}\Big(1+\frac{e(y)+e(-y)}{2}\Big).
$$
Then $f$ takes values in $[0,1]$, $\int f\, d\mu=\frac{1}{2}$, and  for $n=0,1,2,\ldots$ we have
$$
\int f\cdot T^nf \, d\mu =\frac{1}{4}+ \frac{\widehat{\sigma}(n)+\widehat{\sigma}(-n)}{16}=
\Big(\int f\, d\mu\Big)^2+\frac{1}{8}\widehat{\sigma}(n).
$$
\end{proof}

\begin{proof}[Proof of Theorem~\ref{T:S4}]
Combining  Propositions~\ref{T:Correspondence} and \ref{P:01}, we get that for any given   symmetric probability measure $\sigma$ on $\T$, there exist a  system $(Y,\Y,\nu,S)$ and a set $A\in \Y$, with $\nu(A)=\int f\, d\mu$, and such that
$$
\nu(A\cap S^{-n}A)=\nu(A)^2+\frac{1}{8}\cdot \widehat{\sigma}(n) \ \ \text{ for  }  n=0,1,2\ldots.
$$
In order to  complete the proof, we choose the measure $\sigma$ appropriately, as
 in the proof of Corollary~\ref{C:F1}, taking this time into account that it is a probability measure.
\end{proof}

\section{Multiple under and over recurrence}
In this subsection, we prove Theorem~\ref{T:S3}. The next definition gives  a substitute for the notion of a Lebesgue component that is better suited for our multiple recurrence setup.
\begin{definition}
Let $(X,\X,\mu)$ be a probability space and $S$ be a subset of $\Z$. We say that  the sequence $(f_n)_{n\in S}$, of  real valued functions in $ L^\infty(\mu)$, is {\em an orthogonal sequence of order $\ell\in \N$}, if $\int f_{n_1}\cdot f_{n_2}\cdot \ldots \cdot f_{n_\ell}\, d\mu=0$
whenever among the indices  $n_1,\ldots, n_\ell\in S$ there is at least one index  not equal to each of the others.
\end{definition}
\begin{example}
Consider  a Bernoulli $(1/2,1/2)$-system on the sequence space $\{-1,1\}^\Z$. Then  the
sequence of functions $(T^kf)_{k\in \Z}$,  where $f(x)=x(0)$ and $T$ is the shift transformation (defined by $(Tx)(k)=x(k+1)$, $k\in \Z$), is an orthogonal sequence of order $\ell$ for every $\ell\in \N$.
\end{example}
\begin{proposition}\label{P:multiple}
Let $(X,\X,\mu,T)$ be a system and let $\ell\in \N$. Suppose that there exists  $g\in L^\infty(\mu)$ with values in $\{ -1, 1\}$ such that $(T^ng)_{n\geq 0}$ is an orthogonal sequence of order $2\ell$. Then there exists $f\in L^\infty(\mu)$ with values in $[0,1]$ such that for  $d=2,\ldots, \ell+1$ we have that
$$
\int  T^{n_1}f\cdot \ldots \cdot T^{n_d}f\, d\mu <\Big(\int f\, d\mu\Big)^d
$$
for all distinct  $n_1, \ldots, n_d\in \N$. Furthermore, a similar statement holds with the
strict inequality reversed.
\end{proposition}
\begin{proof}
Let $\ell\in \N$ and $g$ satisfies the asserted hypothesis. We can assume that $\norm{g}_{L^\infty(\mu)}\leq 1$.
Let $(a_k)_{k\in \Z}$ be a sequence of real numbers such that
\begin{enumerate}[(i)]
\item
$a_0=0$ and $0<a_n=-a_{-n}<\frac{1}{2^{d+1}\ell!}$ for every $n\in \N$;

 \item
 $\sum_{n=1}^\infty a_n\leq \frac{1}{2}$.
 \end{enumerate}
 Consider the  function
$$
h:=g\cdot \sum_{k\in \Z}^\infty a_k T^kg.
$$
 Then $\norm{h}_\infty\leq \sum_{k\in \Z}|a_k|=2\sum_{k=1}^\infty a_k\leq 1$.
  Note that
$$
\int h \, d\mu=\sum_{k\in \Z}a_k \int g\cdot T^kg\, d\mu=0,
$$
 where the last equality follows from  the order $2$ orthogonality of the sequence $(T^ng)_{n\geq 0}$ and the fact that $a_0=0$.
Let
$$
f:=\frac{1+h}{2}.
$$
Then $f$ takes values in $[0,1]$ and   $\int f\, d\mu=\frac{1}{2}$.
We claim that $f$ satisfies the asserted under-recurrence property. For reader's convenience, we first explain how the argument works for $d=2,3$; the general case is similar but the notation is more cumbersome.
\medskip

\noindent {\bf Proof for $d=2,3$.}
  A simple computation that uses the order $4$ orthogonality of the sequence $(T^ng)_{n\geq 0}$, that $g^2=1$, and the properties of the sequence $(a_k)_{k\in \Z}$,   shows that
$$
\int T^{n_1}h\cdot T^{n_2}h\, d\mu=a_{n_1-n_2}a_{n_2-n_1}=-a_{n_1-n_2}^2
$$
for all distinct $n_1,n_2\in \N$.
 We deduce  that for distinct $n_1,n_2\in \N$ we have
$$
4\int T^{n_1}f\cdot T^{n_2}f\, d\mu=1+2\int h \, d\mu + \int T^{n_1}h\cdot T^{n_2}h\, d\mu
=1-a_{n_1-n_2}^2<1,
$$
where we used that $\int h\, d\mu=0$. Hence, for all distinct $n_1,n_2\in \N$ we have
$$
\int T^{n_1}f\cdot T^{n_2}f\, d\mu<\frac{1}{4}=\Big(\int f\, d\mu\Big)^2.
$$

A similar computation, this time using the order $6$ orthogonality of the sequence $(T^ng)_{n\geq 0}$, shows that
$$
\int T^{n_1}h\cdot T^{n_2}h\cdot T^{n_3}h\, d\mu=a_{n_1-n_2} a_{n_2-n_3} a_{n_3-n_1}+
a_{n_1-n_3} a_{n_2-n_1} a_{n_3-n_2}=0
$$
for all distinct $n_1, n_2, n_3\in \N$.
Furthermore, for distinct $n_1,n_2, n_3\in \N$ we have
\begin{multline*}
8\int T^{n_1}f\cdot T^{n_2}f\cdot T^{n_3}f\, d\mu=1+3\int h \, d\mu +
\int T^{n_1}h\cdot T^{n_2}h\, d\mu+\\ \int T^{n_1}h\cdot T^{n_3}h\, d\mu+ \int T^{n_2}h\cdot T^{n_3}h\, d\mu
+\int T^{n_1}h\cdot T^{n_2}h\cdot T^{n_3}h\, d\mu
\end{multline*}
which is equal to
$$
1-a_{n_1-n_2}^2-a_{n_1-n_3}^2-a_{n_2-n_3}^2<1.
$$
Hence, for all distinct $n_1, n_2, n_3\in \N$ we have
$$
\int T^{n_1}f\cdot T^{n_2}f\cdot T^{n_3}f\, d\mu<\frac{1}{8}=\Big(\int f\, d\mu\Big)^3.
$$
\smallskip

\noindent {\bf Proof for  $d\geq 4$.} First note that since $\int h\, d\mu=0$, we  have that
  \begin{equation}\label{E:d}
2^d \int T^{n_1}f\cdot \ldots\cdot T^{n_d}f\, d\mu=1 +A+B,
\end{equation}
where
\begin{equation}\label{E:0}
A:=\sum_{1\leq i< j\leq d}\int T^{n_i}h\cdot  T^{n_j}h\, d\mu=
-\sum_{1\leq i< j\leq d}a_{n_i-n_j}^2
\end{equation}
and
\begin{equation}\label{E:1'}
B:=\text{ sum of at most $2^d$ terms of the form } \int T^{m_1}h\cdot \ldots\cdot T^{m_{d'}}h\, d\mu
\end{equation}
 where
$d'\in \{3,\ldots,d\}$ and $m_1,\ldots, m_{d'}\in \{n_1,\ldots, n_d\}$ are distinct integers.
Let
\begin{equation}\label{E:2'}
\alpha:=\max_{1\leq i\neq j\leq d}\{|a_{n_i-n_j}|\}.
\end{equation}
Generalizing the computation done  in the case $d=2,3$, this time using the order $2d$ orthogonality of the sequence $(T^ng)_{n\geq 0}$, we get for every $d\geq 2$
and distinct $n_1, \ldots, n_d\in \N$ that
\begin{equation}\label{E:3'}
\int T^{n_1}h\cdot \ldots\cdot T^{n_d}h\, d\mu=\sum_{\pi \in \Sigma[d]} a_{n_1-\pi(n_1)}
\cdot \ldots\cdot  a_{n_d-\pi(n_d)}
\end{equation}
where $\Sigma[d]$ denotes the set of all permutations of the set $\{1,\ldots, d\}$ that have no fixed points.
Combining \eqref{E:1'}, \eqref{E:2'}, \eqref{E:3'}, and using that $|a_n|\leq \frac{1}{2^{d+1}d!}$ for all $n\in \N$,
we get that
$$
|B|\leq 2^d d!\,   \alpha^3  \leq 2^d d!\  \frac{1}{2^{d+1}d!}\, \alpha^2= \frac{\alpha^2}{2}.
$$
Combining this with \eqref{E:0}, we deduce that
$$
1+A+B\leq 1-\alpha^2+\frac{\alpha^2}{2}<1.
$$
Hence,  \eqref{E:d} gives that
$$
\int T^{n_1}f\cdot \ldots\cdot T^{n_d}f\, d\mu<\frac{1}{2^d}=\Big(\int f\, d\mu\Big)^{d}
$$
for all distinct $n_1,\ldots, n_d\in \N$, as required.

A similar (and simpler) argument proves the asserted over-recurrence property. The only change needed is in the definition of the sequence $(a_k)_{k\in \Z}$ we impose that $a_{-n}=a_n$ for every $n\in\N$.
\end{proof}

\begin{proof}[Proof of Theorem~\ref{T:S3}]
Consider a multiple mixing system that satisfies the assumptions of Proposition~\ref{P:multiple} (for example, any
 Bernoulli system). Using Proposition~\ref{T:Correspondence} we get a multiple mixing system and a set satisfying the asserted properties.
\end{proof}

\section{Combinatorial consequences}
In this short section, we  deduce Theorems~\ref{T:c1} and \ref{T:c2}
from their ergodic counterparts.
\begin{proof}[Proof of Theorem~\ref{T:c1}]
Let $(X,\mathcal{X},\mu,T)$ be the mixing system and let $A$ be the set given by the remark following Theorem~\ref{T:S1}. The ergodic theorem guaranties that for some $x_0\in X$ and for every non-negative integer $n$
we have
$$
\lim_{N\to\infty} \frac{1}{N}\sum_{k=1}^Nf_n(T^kx_0)=\int f_n\, d\mu
$$
where $f_n:={\bf 1}_{A\cap  T^{-n}A}$. Let $E:=\{m\in \N\colon T^mx_0\in A\}$.
Then
$$
d(E)=\mu(A) \ \text{ and } \  d(E\cap (E-n))=\mu(A\cap T^{-n}A)\  \text{ for every }\  n\in \N.
 $$
 Hence,
 $$
 d(E\cap (E-n))=\mu(A\cap T^{-n}A)>\mu(A)^2=d(E)^2 \ \text{ for every } \ n\in S_+,
 $$
 and, similarly,
 $$
 d(E\cap (E-n))=\mu(A\cap T^{-n}A)<\mu(A)^2=d(E)^2 \ \text{ for every } \ n\in S_-.
 $$
 Moreover, since the system is mixing, we have that
 $$
 d(E\cap (E-n))=\mu(A\cap T^{-n}A)\to \mu(A)^2=d(E)^2
 $$
 as $n\to \infty$.
\end{proof}
In a similar fashion, we deduce Theorem~\ref{T:c2}  from Theorem~\ref{T:S3}. We include the details for readers convenience.
\begin{proof}[Proof of Theorem~\ref{T:c2}]
For $r\in \N$
let $(X,\mathcal{X},\mu,T)$ be the multiple  mixing system and $A$ be the set given by
Theorem~\ref{T:S3}.
The ergodic theorem guaranties that for some $x_0\in X$ and for all non-negative integers $n_1,\ldots, n_r$
we have
$$
\lim_{N\to\infty} \frac{1}{N}\sum_{k=1}^Nf_{n_1,\ldots, n_d}(T^kx_0)=\int f_{n_1,\ldots, n_r}\, d\mu
$$
where
$f_{n_1,\ldots, n_r}:={\bf 1}_{A\cap  T^{-n_1}A\cap\cdots\cap T^{-n_r}A}$. Let $E:=\{m\in \N\colon T^mx_0\in A\}$. One concludes the proof  of Property $(i)$
exactly as in the proof of Theorem~\ref{T:c1}.
Property $(ii)$ follows in a similar way using the fact that the system is assumed to be multiple mixing.

The existence of a set $E$ that satisfies Property $(i)$ with the strict inequality reversed and also satisfies  Property $(ii)$,
follows in a similar fashion  from Theorem~\ref{T:S3}.
\end{proof}

\section{Under and over recurrent sets in positive entropy systems}
In this section we give explicit constructions of  under and over  recurrent sets in Bernoulli systems and deduce that every system with positive entropy has  under and  over recurrent sets.
\begin{theorem}\label{T:posentropy}
 Every ergodic system with positive entropy has a strictly over-recurrent and a strictly under-recurrent set.
\end{theorem}
\begin{proof}
Suppose that the system  $(X,\X,\mu, T)$ has entropy $h>0$.
It is known \cite{S64} that any Bernoulli shift with entropy smaller than $h$ is a factor of the system $(X,\X,\mu, T)$. Hence, it suffices to show that there exist  Bernoulli shifts with arbitrarily small entropy  that have strictly under and over recurrent sets.

Thus, henceforth, we work with  Bernoulli systems on the space $X:=\{0,1,2\}^\N$ and
  for $i=0,1,2$ we let $p_i:=\mu([i])\in (0,1)$, where  with $[x_1\cdots x_k]$ we denote the cylinder set consisting of those $x\in X$ whose first $k$ entries are $x_1,\ldots,x_k\in \{0,1,2\}$.

We first deal with over-recurrence.
Let
$$
A:=\big\{x\in X\colon  \text{the first non-zero entry of } x \text{ is } 1 \big\}.
$$
Then
$$
A=\bigcup_{n=1}^\infty A_n
$$
where
$$
A_n:=\big\{x\in X\colon \text{the first non-zero entry of } x \text{ is  } 1 \text{ and it is at place } n \big\}.
$$
Since
$
\mu(A_n)=p_0^{n-1}p_1$, we have
$$
\mu(A)=\sum_{n=1}^\infty p_0^{n-1}p_1=\frac{p_1}{1-p_0}=\frac{p_1}{p_1+p_2}=:a.
$$

Moreover, we have
$$
A\cap T^{-n}A=A'_n\cap T^{-n} A
$$
where
$$
A'_n:=(\bigcup_{k=1}^n[(0)_{k-1} 1])\cup [(0)_n].
$$
and $(0)_i$ is used to denote $i$-consecutive zero entries.  Since the set  $A'_n$  depends on the  first $n$ entries of  elements of $X$ only, we have
$$
\mu(A\cap T^{-n}A)=\mu(A'_n)\cdot \mu(T^{-n}A)=a\cdot \mu(A'_n)
$$
where
$$
\mu(A'_n)=p_0^n+\sum_{k=1}^np_0^kp_1=p_0^n+p_1\cdot\frac{1-p_0^n}{1-p_0}=p_0^n+a(1-p_0^n)=
a+p_0^n(1-a).
$$
It follows that
$$
\mu(A\cap T^{-n}A)= a^2+p_0^n(1-a)a>a^2=\mu(A)^2 \ \text{ for every } \ n\in \N.
$$
Hence, the set $A$ is strictly over-recurrent. Note also that by choosing $p_1$ sufficiently close to $1$ (then  $p_0, p_2$ will be close to $0$) we can make the Bernoulli shift  have arbitrarily small entropy.

Next we deal with under-recurrence. We let
$$
A:=\big\{x\in X\colon  \text{the first  two non-zero entries of } x \text{ are } 1 \text{ and }
 2  \text{ in this order}  \big\}.
$$
Then
$$
A=\bigcup_{k,l \geq 0} A_{k,l} \quad \text{ where} \quad   A_{k,l}:=[(0)_k1 (0)_l 2].
$$
Hence,
$$
\mu(A)=\sum_{k,l\geq 0} p_1 p_2 p_0^{k+l}=p_1p_2 \big(\sum_{k\geq 0} p_0^{k}\big)^2=\frac{p_1p_2}{(p_1+p_2)^2}=:a.
$$

Next we fix $n\in \N$ and compute the measure of the set $A\cap T^{-n}A$. We partition the set $A$ into three sets. The first, call it $A_{1}$, consists of  those $x\in A$ whose first two non-zero entries (which are $1$ and $2$)  occur at the first $n$ places. Then
\begin{multline*}
\mu(A_1\cap T^{-n}A)=a\cdot \sum_{0\leq k+l\leq n-2}\mu([(0)_k1(0)_l2])=
ap_1p_2\sum_{k=0}^{n-2}(k+1)p_0^{k}\\=a^2(1+np_0^n-np_0^{n-1}-p_0^n).
\end{multline*}
The second, call it $A_{2}$, consists of  those $x\in A$ whose first two non-zero entries  occur after the first $n$ places. Then
$$
\mu(A_{2}\cap T^{-n}A)=a\cdot \mu([0]_n)=ap_0^n.
$$
The third, call it $A_{3}$, consists of  those $x\in A$ whose first non-zero entry (which is $1$) occurs at the first $n$ places and the second (which is $2$)
occurs after the first $n$ places. Then clearly $A_3\cap T^{-n}A=\emptyset$, hence
$$
\mu(A_{3}\cap T^{-n}A)=0.
$$
Combining the above, we deduce that
\begin{equation}\label{E:identity}
\mu(A\cap T^{-n}A)=a^2(1+np_0^n-np_0^{n-1}-p_0^n)+ap_0^n, \quad n\in \N.
\end{equation}
Then
$$
\mu(A\cap T^{-n}A)<\mu(A)^2 =a^2\Longleftrightarrow n>\frac{p_0(1-a)}{a(1-p_0)}.
$$
So it remains to choose $p_0,p_1,p_2$ so that $p_0<a$;   then the last estimate
will be  satisfied for all $n\in \N$ and   the set  $A$ will be strictly under-recurrent.
We let $p_1=1-s$ and $p_2=ts$ with $s,t\in (0,1)$. Then  the estimate $p_0<a=\frac{p_1p_2}{(p_1+p_2)^2}$
leads to the equivalent estimate
$1-t<\frac{t (1-s)}{(1-s+ts)^2}$ which is satisfied, for example, if $t=\frac{3}{4}$ and  $s<\frac{1}{2}$.

Summarizing, taking $p_0=\frac{1}{4}s$, $p_1=1-s$,  $p_2=\frac{3}{4}s$, we have that for all $s<\frac{1}{2}$ the set $A$, defined  above, is strictly under-recurrent.
Taking $s$ close to $0$ we deduce  the existence  of Bernoulli shifts with arbitrarily small
entropy that have strictly under-recurrent sets. This finishes the proof.
\end{proof}

\section{Singular over-recurrent functions on  a mixing system}\label{SS:ex2}
In Theorem~\ref{T:S2} we showed that there exist  mixing systems with no under-recurrent
sets, and the key to our construction was that a function with singular  spectral measure cannot be under-recurrent. In this section we show that a similar approach cannot be used in order
to construct mixing systems with no over-recurrent functions. We  will show that there exists
a mixing system that has a strictly over-recurrent function with singular spectral measure.

First, we briefly review some basic facts regarding Riesz-products, their proofs can be found in
\cite[pages 5-7]{Qu10} and \cite{Br74, P73, Z32}. If $P_N(t)=\prod_{j=0}^{N-1}(1+a_jcos(3^jt))$, $N\in \N$,  where $(a_j)_{j\geq 0}$ are
real numbers with $|a_j|\leq 1$,  then the sequence of probability measures $(\sigma_N)_{N\in \N}$, defined by $d\sigma_N:= P_N(t)\, dt$, $N\in \N$, converges $w^*$  to a symmetric probability measure
$\sigma$ on $[0,1]$ with Fourier coefficients $\widehat{\sigma}(0)=1$ and
\begin{equation}\label{E:1}
\widehat{\sigma}(n)=
\prod_{j} \big(\frac{a_j}{2}\big), \quad \text{if } n=\sum_{j=0}^k \epsilon_j 3^j,\quad  \epsilon_j=-1,0,1,
\end{equation}
where the product is taken over those $j\in \{0,\ldots, k\}$ for which $\epsilon_j\neq 0$.

The measure $\sigma$ is equivalent to the Lebesgue measure if  $\sum_{j=0}^\infty |a_j|^2<\infty$ and is continuous and  singular if  $\sum_{j=0}^\infty |a_j|^2=\infty$.

We also review some basic facts regarding Gaussian systems; their proofs can be found in \cite[pages 369-371]{CFS82} and \cite[pages 90-92]{G03}. If $\sigma$ is a symmetric probability measure on the circle, then there exist a Gaussian system $(X,\X,\mu,T)$ and a function $f\in L^2(\mu)$ $f$ is a real Gaussian variable, so it is not bounded)   with spectral measure $\sigma$, meaning,   it satisfies
\begin{equation}\label{E:2}
\int f\cdot T^nf\, d\mu=\widehat{\sigma}(n) \ \text{ for }\ n=0,1,2,\ldots.
\end{equation}
The Gaussian system
is mixing if and only if the measure $\sigma$ is  Rajchman, meaning, it satisfies $\widehat{\sigma}(n)\to 0$ as $n\to\infty$. Note that in this case we have  $\int f\, d\mu=0$.

   We proceed now to the  construction (shown to us by B.~Host). We take as $\sigma$  to be the $w^*$-limit  of the sequence of measures $(\sigma_N)_{N\in\N}$ defined by   $$
   d\sigma_N=\prod_{j=0}^{N}\big(1+\frac{\cos(3^jt)}{\sqrt{j+1}}\big)\, dt, \quad N\in \N.
    $$
    Since $\sum_{j=0}^\infty \big(\frac{1}{\sqrt{j+1}}\big)^2=\infty$, as remarked above,
    the measure
 $\sigma$ is singular. Moreover,  it  follows from \eqref{E:1} that  $\widehat{\sigma}(n)> 0$ for every $n\in \N$ and $\widehat{\sigma}(n)\to 0$ as $n\to \infty$.

  Next, we consider a Gaussian system and a function $f\in L^2(\mu)$ that satisfies \eqref{E:2}. As remarked above,  this system is mixing. Moreover,    the function $f$
  has singular spectral measure by construction, and satisfies
 $$
 \int f \cdot T^nf\, d\mu=\widehat{\sigma}(n)> 0\ \text{ for every } n\in \N.
$$
Hence, the function $f$ is strictly over-recurrent, as required.

\end{document}